\documentclass[11pt, reqno]{amsart}

\usepackage{amssymb}
\usepackage{amsmath}
\usepackage{amsthm}
\usepackage{enumerate}
\usepackage{comment}
\usepackage{nicefrac}

\usepackage{xcolor}

\theoremstyle{plain}

\newtheorem{theorem}{Theorem}[section]

\newtheorem{lemma}[theorem]{Lemma}
\newtheorem{proposition}[theorem]{Proposition}

\newtheorem{corollary}[theorem]{Corollary}

\theoremstyle{definition}

\newtheorem{definition}[theorem]{Definition}
\newtheorem{remark}[theorem]{Remark}
\newtheorem{example}[theorem]{Example}

\newcommand{\argmin}{\mathrm{argmin}}
\renewcommand{\leq}{\leqslant}
\renewcommand{\geq}{\geqslant}

\newcommand{\R}{\mathbb{R}}
\newcommand{\N}{\mathbb{N}}

\newcommand{\diam}{\mathrm{diam}}
\newcommand{\epi}{\mathrm{epi}}

\newcommand{\cconv}{\overline{\mathrm{conv}}\:}

\renewcommand{\epsilon}{\varepsilon}
\renewcommand{\phi}{\varphi}
\renewcommand{\tilde}{\widetilde}

\author[C.A.~De~Bernardi]{Carlo Alberto De Bernardi}

\address{Dipartimento di Matematica per le Scienze economiche, finanziarie ed attuariali, Universit\`a Cattolica del Sacro Cuore, Via Necchi 9, 20123 Milano, Italy}

\email{carloalberto.debernardi@unicatt.it}
\email{carloalberto.debernardi@gmail.com}

\author[E.~Miglierina]{Enrico Miglierina
}
\address{Dipartimento di Matematica per le Scienze economiche, finanziarie ed attuariali, Universit\`{a} Cattolica del Sacro Cuore, Via Necchi 9, 20123 Milano, Italy}

\email{enrico.miglierina@unicatt.it}

\author[E.~Molho]{Elena Molho
}
\address{Dipartimento di Scienze economiche e Aziendali, Universit\`{a} degli Studi di Pavia, Via San Felice 5, 27100 Pavia, Italy}

\email{elena.molho@unipv.it}

\author[J.~Somaglia]{Jacopo Somaglia}
\address{Politecnico di Milano, Dipartimento di Matematica, Piazza Leonardo da Vinci 32, 20133 Milano, Italy.}
\email{jacopo.somaglia@polimi.it}



\thanks{The research of the authors has been partially
supported by the GNAMPA (INdAM -- Istituto Nazionale di Alta Matematica). C.A. De Bernardi, E. Miglierina and E. Molho has been partially supported by the Ministry for Science and Innovation, Spanish State Research Agency (Spain),  under project PID2020-112491GB-I00. 
}
\title{Stochastic Approximation in convex multiobjective optimization}

\begin{document}

\begin{abstract}

Given a strictly convex multiobjective optimization problem with objective functions $f_1,\dots,f_N$, let us denote by $x_0$ its solution, obtained as
 minimum point of the linear scalarized problem, where the objective function is the convex combination of $f_1,\dots,f_N$  with weights $t_1,\ldots,t_N$.
The main result of this paper 
gives an estimation of the averaged error that we make if we approximate $x_0$ with the minimum point of the convex combinations of $n$ functions,  chosen among $f_1,\dots,f_N$,  with probabilities $t_1,\ldots,t_N$, respectively, and  weighted with the same coefficient $\nicefrac{1}{n}$.
In particular, we prove that the averaged error considered above converges to 0 as $n$ goes to $\infty$, uniformly w.r.t. the weights $t_1,\ldots,t_N$.  
The key tool in the proof of our stochastic  approximation theorem is a geometrical property, called by us small diameter property, ensuring
 that the minimum point of a convex combination of the function $f_1,\dots,f_N$ continuously depends on the coefficients of the convex combination.
 
\end{abstract}

\maketitle


\markboth{}{}

\section{Introduction}
The main aim of the present paper is to develop a result about the approximation of solutions of a convex multiobjective optimization problem in the spirit of  \cite[Theorem~1, p.28]{BE}. In that paper the authors provided a result whose meaning was explained by the authors themselves (the quotation is translated in English since the original paper is written in French, see \cite[ Remark at p. 30]{BE}):

``This theorem means that to calculate a Pareto Optimum for the functions $f_1,\dots,f_N$, it suffices, approximately, to calculate it for $n$ of them, and, better, most of the choices of $n$ elements among the $N$ will give a good result."

In order to make clear the comment above, we briefly describe what is a multiobjective optimization problem. For a complete overview about this field we refer the reader to the, now classical, monographs \cite{Luc} and \cite{Jahn}. Let $f_1,\dots,f_N$ be $N$ functions from a normed space $X$ to $\R$. A point $x_0\in X$ is a Pareto Optimum point (or efficient point) for $f_1,\dots,f_N$ when there is no $x \in X$ such that $f_k(x) \leq f_k(x_0)$ for $k = 1,\dots, N$ and $f_i(x) <f_i(x_0)$ for some $i \in \{1,\dots, N\}$. One of the most common technique to find the Pareto Optimum points is the so called linear scalarization, i.e., to solve a family of scalar optimization problems where the objective function is given by a convex combination of the functions $f_1,\dots,f_N$ (see, e.g.,  \cite[Ch.4]{Luc} and  \cite[Ch.5]{Jahn}). In our paper we deal with a special case of linear scalarization for a multiobjective optimization problem where, under convexity assumptions on the functions $f_1,\dots,f_N$, each Pareto Optimum is completely characterized as minimum point of a convex combinations of the functions $f_1,\dots,f_N$. Namely, if the functions $f_1,\dots,f_N$ are strictly convex, a point $x_0$ is a Pareto Optimum point for $f_1,\dots, f_N$ if and only if there exists $t=(t_1,\dots, t_N)\in \Sigma_N:=\{t\in [0,1]^N\colon\sum_{i=1}^N t_i=1\}$ such that  $x_0 $ is a minimum point for the scalar function $u_t(x)=\sum_{k=1}^{N}t_kf_k(x)$ (see  \cite[Lemma~1, p. 27]{BE}).  

The relevance of linear scalarization in the theory and practice of multiobjective optimization motivates the study of the properties of a convex combinations of functions. Moreover, as we will see, a key tool 
to obtain a result in the spirit of \cite{BE}, will be a geometrical property ensuring
 that the minimum point of a convex combination of continuous strictly convex function $f_1,\dots,f_N$ continuously depends on the coefficients of the convex combination.

The approach developed by Enflo and Beauzamy in \cite{BE} is deeply original with respect to the field of multiobjective optimization and gives an unusual viewpoint about the approximation of solution set. Despite these interesting features, the paper \cite{BE} seems not to have been widely considered in the literature about multiobjective optimization. For this reason, we are interested in  studying this result and try to weaken its assumptions. The key point of the approach of \cite{BE} is to study the Pareto Optimum points of a convex multiobjective optimization problem with objective functions $f_1,\dots,f_N$, by identifying these points with the solutions of minimum points of convex combinations of $f_1,\dots,f_N$. In particular, we need to ensure the continuity of the map $\phi(t)=\argmin(u_t)$ where $t \in \Sigma_N$. Moreover, in view of linear scalarization procedure, this type of result is interesting in itself.

In the present paper we restrict our attention to the case where the functions $f_1,\dots,f_N$ are such that, for every $k=1,\dots,N$:
\begin{enumerate}[(a)]
\item $f_k$ is a continuous strictly convex function that is bounded on bounded sets;
\item $f_k$ is coercive (i.e. for each $C\in \R$ the set $\{x\in X\colon f(x)\leq C\}$ is bounded).
\end{enumerate}
Under these assumptions, it is quite straightforward to show that the function $\phi$ is continuous whenever $X$ is a finite dimensional normed space. On the other hand, it is not always possible to prove the continuity of the map $\varphi$ when $X$ is infinite-dimensional. Indeed, we provide an example of two functions defined on the Hilbert space $\ell_2$, where the corresponding function $\phi$ is not continuous, even if the assumptions (a)-(b) above hold (see Example \ref{e: argmindiscontinuo}).  The continuity of $\phi$ can be obtained by adding a property, already introduced in \cite{C1975} and \cite{VolleZali} under different terminologies, that here we call small diameter property. 
Section \ref{sec: notation}  is mainly devoted to the study of small diameter property. It is worth pointing out that in Proposition \ref{t: liftingnorms} we prove a result that allows to lift strongly exposed point from the norm of $x$ to a function $f:X\rightarrow \R$, hence providing a good tool to find examples of functions satisfying small diameter property.  In Section \ref{sec: nonstronglyexposed}, it is shown that $\phi$ is continuous if the functions $f_1,\dots,f_N$ satisfy the small diameter property and additionally properties (a)-(b) mentioned above. Finally, Section \ref{sec: generalizzazioneEB} provides a  ``stochastic" approximation of Pareto Optimum points for $f_1,\dots,f_N$ not depending on the choice of the value $t$ associated to each Pareto Optimum (Theorem \ref{t: mainApplication}). This theorem can be applied to a larger setting than the result of \cite{BE}, as proved in Remark \ref{rem:costruzioneCarloJacopo}, where, by using special norms built up in \cite{DEBESOM-ALUR}, we prove that there exists a set of functions $f_1,\dots,f_N$ not satisfying the assumptions of   \cite[Theorem~1, p.28]{BE} but enjoying the conditions required by our result. Moreover, an example shows that the small diameter property assumption cannot be dropped.

\section{Notation and preliminaries}\label{sec: notation}

In this section we introduce some notions and basic results that we shall need in the sequel of the paper. The section is divided in two parts: the former is devoted to the study of geometrical properties of convex sets and functions, the latter to a brief introduction to  multiobjective optimization.
We consider only nontrivial real normed linear spaces. 
If $X$ is a normed space with topological dual  $X^*$, then  $B_X$, $U_X$, and $S_X$ are the closed unit ball, the open unit ball, and the unit sphere of $X$, respectively. We refer to \cite{ALKA}, \cite{FHHMZ} and \cite{PhelpsBook} for unexplained notation and terminology.

\subsection{Geometrical properties of convex sets and functions}

Let $D$ be a  convex subset of a normed space $X$. A point $x\in D$ is \textit{supported} by a linear functional $x^*\in X^*\setminus\{0\}$ if $x^*(x)=\sup_D x^*$. A point $x\in D$ is called an \textit{exposed point} of $D$ if $x$ is supported by $x^*\in X^*$ and $\{y\in D\colon x^*(y)=x^*(x)\}=\{x\}$. A point $x\in D$ is  \textit{strongly exposed} by $x^*\in X^*$ if it is supported by $x^*$ and $x_n\to x$ for all sequences $\{x_n\}\subset D$ such that $\lim_{n\to \infty}x^*(x_n)=\sup_{D}x^*$. In this case, $x$ is also strongly exposed by $\lambda x^*$, for any $\lambda>0$. Moreover, a point $x\in D$ is strongly exposed by $x^*\in X^*\setminus\{0\}$ if and only $x^*$ is bounded over $D$ and  $\diam\, S(x^*,\delta,D)\to 0$ as $\delta\to 0$, where 
\begin{equation*}
{S(x^*,\delta,D):=\{y\in D\colon x^*(s)>\sup_{D} x^*-\delta\}.    }
\end{equation*}

Let  $f:D\to\R$ be a convex continuous function. By $$\epi(f):=\{(x,t)\in D\times \R\colon t\geq f(x)\}$$ we denote the \textit{epigraph} of the function $f$. Moreover, $\partial f(x)$ denotes the \textit{subdifferential} of $f$ at $x$. A point $(x,f(x))\in\epi(f)$ is supported by $(x^*,-1)$ if and only if $x^*\in\partial f(x)$.  Finally, by $\R^+$ we denote the interval $(0,\infty)$.

\begin{definition}
Let $D$ be a convex subset of $X$,  $x\in D$, and $f\colon D\to\R$.
\begin{itemize}
    \item Let $x^*\in\partial f(x)$, we say that $x$ is an \textit{$x^*$-small diameter point} for $f$ if each sequence $\{x_n\}\subset D$, satisfying  $x^*(x_n)-f(x_n)\to x^*(x) - f(x)$, is norm convergent.
    \item If, for each $x^*\in \partial f(x)$, the point $x$ is an $x^*$-small diameter point for $f$,  then
we say that  $x$ is a \textit{small diameter point} for $f$.
\item $f$ has the \textit{small diameter property} (SDP) if each $y\in D$ is a small diameter point for $f$.
\end{itemize}
\end{definition}

We recall the following result which relates small diameter points for $f$  in $X$ with strongly exposed points of $\epi(f)$ in $X\times \R$.

\begin{lemma}[\cite{C1975}]\label{lemma:SDPperepi}
Let $D$ be a convex subset of $X$, $f\colon D\to \R$ a continuous convex function, $x\in X$, and $x^*\in \partial f(x)$. Then $x$ is an $x^*$-small diameter point for $f$  if and only if $(x,f(x))$ is a strongly exposed point of $\epi(f)$ by $(x^*,-1)$. 
\end{lemma}

We deduce that if $f\colon X\to \R$ has the SDP, then for each $x\in X$, $(x,f(x))$ is supported by $(x^*,-1)$ if and only if $(x,f(x)) \in \epi(f)$ is strongly exposed by $(x^*,-1)$. We are going to show an elementary fact about the SDP;  we shall need this result  in the next section (cf. the proof of \cite[Theorem~4.2]{VolleZali}).  

\begin{proposition}\label{p: sumsdp}
Let $f,g\colon X\to \R$ be  convex  continuous functions. If $f$ has the SDP then $f+g$ has the SDP.
\end{proposition}

\begin{proof}
Let $x_0\in X$ and $x^*\in \partial(f+g)(x_0)$. We are going to show that $x_0$ is a $x^*$-small diameter point for $f+g$. Indeed, since $f+g$ is a convex function and $x^*\in \partial(f+g)(x_0)$, we have
\begin{equation*}
x^*(y)-(f+g)(y)\leq x^*(x_0)-(f+g)(x_0),
\end{equation*}
for any $y\in X$. Now, let $\{x_n\}\subset X$ be a sequence such that $x^*(x_n)-(f+g)(x_n)$ converges to $x^*(x_0)-(f+g)(x_0)$. By \cite[Theorem 3.23]{PhelpsBook} there exists $x_1^*\in \partial f(x_0)$ and $x_2^*\in \partial g(x_0)$ such that $x^*=x_1^*+x_2^*$. Fix  $\epsilon>0$, then eventually  we have
\begin{equation*}
\begin{split}
\varepsilon&> x^*(x_0) - (f+g)(x_0) -x^*(x_n) + (f+g)(x_n)\\
&= x_1^*(x_0)- f(x_0) - x_1^*(x_n) + f(x_n) + x_2^*(x_0)- g(x_0) - x_2^*(x_n) + g(x_n)\\
&\geq x_1^*(x_0)- f(x_0) - x_1^*(x_n) + f(x_n) \geq 0.
\end{split}
\end{equation*}
Since $f$ has the SDP,  the previous chain of inequalities and the arbitrariness of $\epsilon>0$ show that $\{x_n\}$ converges to $x_0$. Hence, $f+g$ has the SDP.
\end{proof}

The next result shows, under suitable hypotheses, how to lift strongly exposed points from the norm of a normed space $X$ to the graph of a function $f\colon X\to \R$. It provides a good tool for constructing examples of functions having the SDP.

\begin{proposition}\label{t: liftingnorms}
Let $(X,\|\cdot\|)$ be a normed space, $x_0\in S_X$, and suppose that $x_0\in B_X$ is strongly exposed by $x^*\in S_{X^*}$. 
Let $h\colon [0,\infty)\to\R$ be a continuous convex strictly increasing function and let $\lambda\in\R^+$ and $\lambda^*\in\R$, suppose that $\lambda$ is a $\lambda^*$-small diameter point  for $h$. Then:
\begin{enumerate}
\item $h(\|\cdot\|)$ is a continuous convex  function on $X$;
\item $\lambda x_0$ is a $\lambda^* x^*$-small diameter point for $h(\|\cdot\|)$;
\item the origin of $X$ is a $0$-small diameter point for $h(\|\cdot\|)$, where $0\in X^*$ denotes the null functional.
\end{enumerate}  
\end{proposition}

\begin{proof} The first point is well-known and easy to prove. Let us prove (ii). Clearly, since $h$ is strictly increasing, $\lambda^*$ is positive. Hence, for each $x\in X$, we have
\begin{equation}\label{eq: SE}
\lambda^* x^*(x)-h(\|x\|)\leq \lambda^* (x^*(x)-\|x\|)+\lambda^*\lambda-h(\lambda)\leq\lambda^* x^*(\lambda x_0)-h(\|\lambda x_0\|),
\end{equation}
where the first inequality holds since
$\lambda^*\in\partial h(\lambda)$. Now, let us consider a sequence $\{x_n\}\subset X$ and suppose that 
$$\lambda^* x^*(x_n)-h(\|x_n\|)\to\lambda^* x^*(\lambda x_0)-h(\|\lambda x_0\|).$$
Since, for $n\in\N$, 
$$\lambda^* x^*(x_n)-h(\|x_n\|)\leq \lambda^* \|x_n\|-h(\|x_n\|)\leq\lambda^* \lambda-h(\lambda),$$
we have that $\|x_n\|\to\lambda$, in particular $\|x_n\|>0$ eventually.
Moreover, by \eqref{eq: SE}, we have:
$\lambda^* (x^*(x_n)-\|x_n\|)\to 0$ and hence $x^*\left(\frac{x_n}{\|x_n\|}\right)\to 1$. Since $x_0\in B_X$ is strongly exposed by $x^*\in S_{X^*}$, we have $\frac{x_n}{\|x_n\|}\to x_0$, and hence ${x_n}\to \lambda x_0$. The proof of (ii) is completed.

In order to prove  (iii), we can suppose without any loss of generality that $h(0)=0$. Then, it  is sufficient to observe that, by the continuity of the inverse function of $h$, if $h(\|x_n\|)\to h(\|0\|)=0$ then $\|x_n\|\to 0$. 
\end{proof}

A function $f\colon X\to \R$ is said \textit{strictly convex} on a convex set $D\subset X$ if, for every $x,y\in D$, the following relation holds: $f(\frac{x+y}{2})<\frac{1}{2}(f(x)+f(y))$. We now recall the stronger notion of uniformly convex function. 

\begin{definition}
Let $D$ be a convex subset of a normed space $X$, $x_0\in D$ and $f\colon D\to \R$. Then we say that
\begin{itemize}
\item \textit{$f$ is uniformly convex at $x_0$} if for every $\varepsilon>0$ there exists $\delta>0$ such that for every $y\in D$ such that $\|x_0-y\|\geq\varepsilon$ it holds
\begin{equation*}
 f\left(\frac{x_0+y}{2}\right)\leq \frac{1}{2}f(x_0) + \frac{1}{2}f(y) - \delta.
\end{equation*} 
\item \textit{$f$ is uniformly convex on $D$} if for every $\varepsilon>0$ there exists $\delta>0$ such that for every $x,y\in D$ such that $\|x-y\|\geq\varepsilon$ it holds
\begin{equation*}
f\left(\frac{x+y}{2}\right)\leq \frac{1}{2}f(x) + \frac{1}{2}f(y) - \delta.
\end{equation*} 
\end{itemize}
\end{definition}

\begin{remark}
    The following implications hold.
    \begin{enumerate}
        \item If a function $f\colon X\to \mathbb{R}$ is uniformly convex on a convex subset $D\subset X$, then it is uniformly convex in $x$, for each $x\in D$.
        \item If $f\colon X\to \mathbb{R}$ is uniformly convex at $x$, for each $x\in X$, then $f$ has the SDP.
        \item If $f\colon X\to \R$ has the SDP then it is strictly convex.
        \end{enumerate}
\end{remark}
  \begin{proof}
      The  implication contained in (i) is trivial, let us prove (ii). Suppose on the contrary that $f$ has not the SDP, then there exist $x\in X$, $x^*\in \partial f(x)$, $\varepsilon>0$, and a sequence $\{x_n\}\subset X$ such that $x^*(x_n)-f(x_n)$ goes to $x^*(x)-f(x)$, when $n\to +\infty$,  but $\|x_n-x\|\geq \epsilon$, whenever $n\in\N$. Since $f$ is uniformly convex in $x$, there exists $\delta>0$ such that 
\begin{equation*}
 f\left(\frac{x+x_n}{2}\right)\leq \frac{1}{2}f(x) + \frac{1}{2}f(x_n) - \delta,
\end{equation*}
for each $n\in \N$. From which we get
\begin{equation*}
 2\delta + x^*(x) - f(x) + x^*(x_n) - f(x_n) \leq  2x^*\left(\frac{x+x_n}{2}\right)-2f\left(\frac{x+x_n}{2}\right).
\end{equation*}
Hence, for sufficiently large $n$ we get
\begin{equation*}
    \delta + 2[x^*(x) - f(x)] \leq 2\left[x^*\left(\frac{x+x_n}{2}\right)-f\left(\frac{x+x_n}{2}\right)\right].
\end{equation*}
The last inequality contradicts the fact that the point $(x,f(x))\in \epi(f)$ is supported by the functional $(x^*,-1)$.

Finally, let us prove (iii). If $f\colon X\to \R$ is not strictly convex, the its graph contains a segment of length $\ell>0$, therefore one can find a point $x\in X$ and $x^*\in X^*$ for which $\diam(S(y^*,\delta,\epi(f)))\geq \ell$, where $y^*=(x^*,-1)$. Therefore $f$ has not the SDP. 
  \end{proof} 

In general, none of the implications contained in the previous remark  can be reversed (this is easy to see for (i) and (ii), and an example of strictly convex function without the SDP is contained in Example~\ref{e: argmindiscontinuo} below). On the other hand, the next proposition, whose proof is left to the reader, shows that strict convexity and uniform convexity coincide on compact sets. 

\begin{proposition}\label{p: scimpliesuc}
Let $K$ be a compact and convex subset of a normed space $X$. Let $f\colon K\to \R$ be a continuous and strictly convex function. Then $f$ is uniformly convex on $K$.
\end{proposition}

\subsection{Convex multiobjective optimization problem}

 Multiobjective optimization is a topic of optimization theory that finds many applications in Economics, Operations Research and Engineering. Here we just give a brief account on the notions of solution and on the linear scalarization procedure. For a complete overview about this topic, we refer the readers to the, now classical, monographs \cite{Luc} and \cite{Jahn}. Moreover, we quote also \cite{Ehrgott}, even if this book is devoted only to the finite dimensional setting. Even if the results we mention in this subsection can be proved in an easy way, their proofs can be found in the above mentioned books, often in a more general framework. 

Let $f_1,\dots, f_N$ be $N$ real valued functions defined on a normed linear space $X$.

\begin{definition}
    A point $x_0 \in X$ is a \textit{Pareto Optimum point} for $f_1,\dots, f_N$ when there is no $x \in X$ such that $f_k(x) \leq f_k(x_0)$ for each $k=1, \dots, N$ and $f_i(x)<f_i(x_0)$ for some $i\in \{1, \dots, N$\}. 
\end{definition}

A weaker notion of solution for a multiobjective optimization problem is the notion of weak Pareto Optimum point.

\begin{definition}
    A point $x_0 \in X$ is a \textit{weak Pareto Optimum point} for $f_1,\dots, f_N$ when there is no $x \in X$ such that $f_k(x) < f_k(x_0)$ for each $k=1, \dots, N$.
\end{definition}

Of course, each Pareto Optimum point is also a weak Pareto Optimum point, but in general the reverse implication does not hold.
Nevertheless, the two notions coincide under suitable convexity assumptions. Namely, the following result holds.
\begin{proposition}\label{prop:weakPareto=Pareto}
    Let $f_1,\dots,f_N$ be strictly convex functions then each weak Pareto Optimum Points for $f_1,\dots, f_N$ is also Pareto Optimum for $f_1,\dots, f_N$.
\end{proposition}

A multiobjective optimization problem consists in finding  the (weak) Pareto Optimum points for objective functions $f_1,\dots, f_N$. It is easy to see that, in general, Pareto Optimum points are not unique. A common approach to solve a multiobjective optimization problem is to reduce it to a parametrized family of scalar optimization problems (this procedure is called scalarization of a multiobjective optimization problem). Many different scalarization methods are developed in the literature (see the monographs quoted above and the references therein). Among them, one of the most used is the linear scalarization or weighted sum scalarization.
Let us denote by $\Sigma_N$ the $(N-1)$-dimensional simplex, i.e., $\Sigma_N=\{t\in [0,1]^N:\sum_{i=1}^{N}t_i=1\}$ and, for $t=(t_1,\ldots,t_N)\in\Sigma_N$, let $u_t(x)=\sum_{i=1}^{N}t_if_i(x)$. We have the following sufficient optimality condition.

\begin{proposition}\label{prop:sufficient optimality}
    Let $x_0 \in X$. If there exists $\Tilde{t} \in \Sigma_N$ such that $u_{\Tilde{t}}(x_0)\leq u_{\Tilde{t}}(x)$ for every $x\in X$, then $x_0$ is a weak Pareto Optimum point for $f_1,\dots, f_N$.
\end{proposition}

The previous sufficient condition becomes also necessary if we add convexity assumptions for the objective functions.

\begin{proposition}\label{prop:necessary optimality}
    Let $x_0 \in X$ and $f_1,\dots,f_N$ be convex functions. If $x_0$ is a weak Pareto Optimum point for $f_1,\dots, f_N$, then there exists $\Tilde{t} \in \Sigma_N$ such that $u_{\Tilde{t}}(x_0)\leq u_{\Tilde{t}}(x)$ for every $x\in X$.
\end{proposition}

In the sequel of the paper we always consider objective functions $f_1,\dots,f_N$ that are  strictly convex. Therefore, by combining Propositions \ref{prop:weakPareto=Pareto}, \ref{prop:sufficient optimality} and \ref{prop:necessary optimality}, we obtain the following remark, that plays a key role in our setting.

\begin{remark} \label{rem:multiobjective strictly convex functions}
    Let us consider a multiobjective optimization problem with strictly convex objective functions $f_1, \dots, f_N$. Then a point $x_0 \in X$ is a Pareto Optimum point for $f_1, \dots, f_N$ if and only if $x_0$ is a minimum point for the function $u_{t}(x)$ for some $t \in \Sigma_N$. 
\end{remark}

\section{Continuity of $\mathrm{argmin}$ map for a convex combinations of functions}\label{sec: nonstronglyexposed}

 This section is devoted to find some conditions ensuring that the minimum point of a convex combination of continuous strictly convex function $f_1,\dots,f_N$ continuously depends on the coefficients of the convex combination. In order to provide existence of minima we restrict our attention to reflexive spaces.
 
\begin{definition}\label{def: Ufunct}
Let $X$ be a reflexive Banach space. We say that a  function $f\colon X\to \R$ is an \textit{U-function} if it satisfies the following conditions:
\begin{enumerate}[(a)]
\item $f$ is a continuous strictly convex function that is bounded on bounded sets.
\item $f$ is coercive (i.e. for each $C\in \R$ the set $\{x\in X\colon f(x)\leq C\}$ is bounded).
\end{enumerate}
\end{definition}

Let us notice that \textit{U}-functions have a unique minimum point. Indeed, let $f$ be an \textit{U}-function and take $C\in \R$ in such a way that the set  $S_C=\{x\in X\colon f(x)\leq C\}$ is non-empty. Since $f$ is continuous and convex, $S_C$ is closed and convex, therefore it is a weakly closed subset of $X$, hence $S_C$ is weakly compact. Since $f$ is a convex function, it is weakly lower semicontinuous, therefore it admits a minimum point in $S_C$. Uniqueness follows by strict convexity of $f$.

Let $f_1,\dots, f_N$ be a finite family of \textit{U}-functions. For 
\begin{equation*}
{t=(t_1,\dots, t_N)\in \Sigma_N,}
\end{equation*}
we recall that $u_t\colon X\to \R$ is the function defined by $u_t = \sum_{i=1}^N t_i f_i$. Since, for each $t\in\Sigma_N$, $u_t$ is an \textit{U}-function, we can consider the map $\phi\colon \Sigma_N\to X$ defined by 
\begin{equation*}
\phi(t):=\argmin(u_t).
\end{equation*}
 Let us stress the fact that the map $\phi$ implicitly depends on the choice of the functions $f_1,\dots,f_N$. 
 
\begin{remark}\label{r: minarebdd}
For a family of  \textit{U}-functions $f_1,\ldots, f_N$ as above, consider the set $\mathcal{M}=\phi(\Sigma_N)$ containing all  minimun points of the functions $u_t$ ($t\in \Sigma_N$). Proceeding as in  \cite[p. 28]{BE},  it is not difficult to see that the set $\mathcal{M}$ is bounded. 
\end{remark}
\begin{remark}[{\bf Multiobjective optimization with \textit{U}- functions}] \label{rem: multiobjective U functions} 
When we consider a convex multiobjective optimization problems where the objective functions are \textit{U}-functions, by Remark \ref{rem:multiobjective strictly convex functions} it follows that the set  $\mathcal{M}=\phi(\Sigma_N)$ is the set of all Pareto Optimum points for the functions $f_1,\dots,f_N$. 
\end{remark}

 If $X$ is finite-dimensional, for any choice of  \textit{U}-functions $f_1,\ldots,f_N$,  it is possible to show, by a standard compactness argument, that the function $\phi\colon \Sigma_N \to X$ is continuous (here, $\Sigma_N$ is endowed with the standard product topology). Let us point out that, in the infinite-dimensional case, this is not true in general:  Example~\ref{e: argmindiscontinuo} below  contains two  \textit{U}-functions $f_1,f_2$, for which the corresponding function $\phi$ is not continuous. 
  The main aim of the present section is to provide suitable geometric assumptions, that ensure the continuity of the function $\phi$ in the infinite-dimensional setting (see Theorem~\ref{t: continuityargmin}, below).

We start by proving the continuity of the map $\phi$ at a given point of $\Sigma_N$. In the sequel, for a positive real number $\alpha$ and a \textit{U}-function $f$, we denote by $S_{\alpha}(f)$ the strict sublevel set of $f$ at level $(\min f +\alpha)$, i.e., $$S_{\alpha}(f)=\{x\in X\colon f(x)<\min (f)+\alpha\}.$$ 

\begin{lemma}\label{l: pointwisecont}
Let $f_1,\dots, f_N$ be a finite family of \textit{U}-functions and  $\hat{t}\in \Sigma_N$. Suppose that $\diam(S_\alpha(u_{\hat{t}}))$ goes to zero for $\alpha$ going to zero. Then the map $\phi\colon \Sigma_N\to X$ is continuous at $\hat{t}$.
\end{lemma}

\begin{proof} Let us denote  $x_0=\phi(\hat{t})=\argmin(u_{\hat{t}})$.
We claim that if a sequence $\{v^n\}\subset \Sigma_N$ converges to $\hat{t}$,  then the sequence of functions $\{u_{v^n}\}$ converges uniformly to $u_{\hat{t}}$ on  bounded sets. Indeed, let $D\subset X$ be a bounded set and $x\in D$, denote $\hat{t}=(\hat{t}_1,\ldots,\hat{t}_N)$ and $v^n=(v^n_1,\ldots,v^n_N)$, then 
\begin{eqnarray*}
|u_{\hat{t}}(x)-u_{v^n}(x)|&=&\left|\sum_{i=1}^N \hat{t}_i f_i(x)-\sum_{i=1}^N v_i^n f_i(x)\right|\\
&\leq& \max_{i}\sup_{y\in D}|f_i(y)| \sum_{j=1}^N|\hat{t}_j-v^n_j|.
\end{eqnarray*}
Taking the supremum over $x\in D$ and letting $n$ to infinity we get the claim. 

Now suppose that the function $\varphi\colon\Sigma_N\to X$ is not continuous at $\hat{t}$. Then there exist a sequence $\{t_n\}\subset\Sigma_N$, converging to $\hat{t}$, and a positive real number $\xi>0$ such that $\|x_n-x_0\|>\xi$, where $x_n:=\phi(t_n)$ for every $n\in\N$. Since $\diam(S_{\alpha}(u_{\hat{t}}))$ goes to zero, there exists $\alpha>0$ such that $\diam(S_{\alpha}(u_{\hat{t}}))<\xi/2$. By Remark \ref{r: minarebdd}, the set $\mathcal{M}$ is bounded. Since $\{u_{t_n}\}$ converges uniformly to $u_{\hat{t}}$ on $\mathcal{M}$, there exists $n\in\N$ such that $|u_{t_n}(x_0)-u_{\hat{t}}(x_0)|<\alpha/3$ and $|u_{t_n}(x_n)-u_{\hat{t}}(x_n)|<\alpha/3$. Since $x_0\in S_{\alpha}(u_{\hat{t}})$ and $\|x_n-x_0\|>\xi$, we have $u_{\hat{t}}(x_n)\geq\alpha + u_{\hat{t}}(x_0)$. Hence, we obtain
\begin{equation*}
u_{t_n}(x_n)=u_{t_n}(x_n)-u_{\hat{t}}(x_n)+u_{\hat{t}}(x_n)>2\alpha/3 + u_{\hat{t}}(x_0)>\alpha/3+ u_{\hat{t}}(x_0)>u_{t_n}(x_0).
\end{equation*}
A contradiction to the fact that $x_n=\phi(t_n)=\argmin(u_{t_n})$.
\end{proof}

\begin{theorem}\label{t: continuityargmin}
Let $f_1,\dots,f_N$ be a finite family of \textit{U}-functions. The following assertions hold
\begin{enumerate}
    \item If $f_i$ has the SDP for some $i\in\{1,\dots, N\}$, then the map $\varphi\colon \Sigma_N\to X$ is continuous on $\{t=(t_1,\ldots,t_n)\in \Sigma_N\colon t_i\neq 0\}$.
    \item If $f_i$ has the SDP for every $i=1,\dots, N$, then the map $\varphi\colon \Sigma_N\to X$ is continuous on $\Sigma_N$.
    \end{enumerate} 
\end{theorem}

\begin{proof}
 Let us prove (i),  (ii)  follows immediately. Suppose that $t=(t_1,\ldots,t_n)\in \Sigma_N$ is such that $t_i\neq 0$, and let us prove that $\phi$ is continuous at $t$. By Proposition \ref{p: sumsdp}, the function $u_t$ is a \textit{U}-function which satisfies the SDP. If $x_t=\phi(t)$, by Lemma~\ref{lemma:SDPperepi}, we have that $(x_t,u_t(x_t))$ is strongly exposed by $(0,-1)\in X^*\times\R$. Hence, $\diam(S_\alpha(u_{t}))$ goes to zero for $\alpha$ going to zero. The assertion follows by Lemma \ref{l: pointwisecont}.
\end{proof}

\begin{remark}
Let us point out that if a continuous convex function $f \colon X\to\R$ has the SDP then, by \cite[Proposition~3.13]{VolleZali}, it is coercive. Therefore, under this additional assumption, condition (b) in Definition~\ref{def: Ufunct} is redundant.    
\end{remark}

The following example shows that the map $\phi\colon \Sigma_N\to X$ is not always continuous, even if  $f_1,\dots,f_N$ are \textit{U}-functions. Therefore, this example points out the key role played in Theorem \ref{t: continuityargmin} by SDP. 

\begin{example}\label{e: argmindiscontinuo}
Let $X=\ell_2$, $h_1,h_2\colon X\to \R$, and $g\colon \R\to \R$ be defined by
\begin{equation*}
 h_1(x)=\sum_{n=2}^{+\infty}\frac{(x_n-x_1)^2}{2^n}, ,\qquad\qquad x=(x_1,x_2,x_3,\dots)\in\ell_2,
\end{equation*}
\begin{equation*}
  h_2(x)=\sum_{n=2}^{+\infty}\frac{(x_n+x_1-1)^2}{2^n},\qquad\qquad x=(x_1,x_2,x_3,\dots)\in\ell_2,
\end{equation*}

\begin{equation*}
g(y)= \begin{cases}
(y-2)^2\,\, & \text{if } y\geq 2,\\
0 \,\, & \text{if } 0\leq y<2.
\end{cases}
\end{equation*}
Then define 
\begin{equation*}
 f_1(x)=h_1(x) + g(\|x\|),\quad  f_2(x)=h_2(x) + g(\|x\|)\qquad\qquad (x\in\ell_2).
\end{equation*}
Then $f_1$ and $f_2$ are   real-valued \textit{U}-functions  on $\ell_2$ such that:
\begin{enumerate}[(1)]
\item 
 $f_1$ attains its minimum at $x=0$;
 \item for each $t\in (0,1)$, the  function $u_t:=(1-t)f_1+t f_2$ attains its minimum outside $2U_X$ (i.e. $\|\phi(t)\|\geq2$).
 \end{enumerate}
 In particular, the map $\phi\colon \Sigma _2\to X$ is not continuous at $(t_1,t_2)=(1,0)\in \Sigma_2$.\end{example}

\begin{proof}
    Let us first prove that $f_1$ and $f_2$ are \textit{U}-functions. Since, for each $x=(x_1,x_2,x_3,\dots)\in \ell_2$, we have $$h_2\left((x_1,x_2,x_3,\dots)\right)=h_1\left((1-x_1,x_2,x_3,\dots)\right),$$
    it is sufficient to prove it just for $f_1$. 
    
    Continuity of  $f_1$  is trivial, while condition (b), in Definition~\ref{def: Ufunct}, follows by the definition of the function $g$. 
    We claim that $h_1$ is bounded on bounded sets. Indeed, if $D$ is a bounded set and $x\in D$, we get
\begin{equation*}
|h_1(x)|=\left|\sum_{n=2}^{+\infty}\frac{(x_n-x_1)^2}{2^n}\right|\leq \sup_{n\in\N}|x_n-x_1|^2 \sum_{n=2}^{+\infty}\frac{1}{2^n},
\end{equation*}
    and the claim is proved. By our claim and by our construction, $f_1$ is bounded on bounded sets.  It remains to show that $f_1$ is strictly convex. Since $g(\|\cdot\|)$ is a convex function, it is enough to show that $h_1$ is strictly convex. 
    Let $x,y\in \ell_2$ be distinct,  $\lambda\in (0,1)$, and let us prove that 
 \begin{equation*}
 \theta_1 := h_1(\lambda x +(1-\lambda )y) - [\lambda h_1(x) + (1-\lambda)h_1(y)] < 0.
 \end{equation*}
    To do this, let us consider the function $k_1\colon\ell_2\times \ell_2 \to \R$ defined by
\begin{equation*}
 k_1(x,y)=\sum_{n=2}^{+\infty} \frac{(x_n-x_1)(y_n-y_1)}{2^n},\quad x=(x_1,x_2,\dots),y=(y_1,y_2,\dots)\in \ell_2, 
\end{equation*}
and observe that
\begin{enumerate}
\item $h_1(x)>0$ if and only if $x\neq0$;
\item $h_1(x\pm y)=h_1(x)+ h_1(y)\pm 2k_1(x,y)$.
\end{enumerate}
 Hence, by using (i) and (ii), we get
\begin{equation*}
\begin{split}
\theta_1 &= \sum_{n=2}^{\infty}\frac{(\lambda x_n - \lambda x_1 +(1-\lambda)y_n - (1-\lambda)y_1)^2}{2^n} -\lambda h_1(x) -(1-\lambda) h_1(y)\\
&= \lambda^2 h_1(x) +(1-\lambda)^2 h_1(y) + 2\lambda (1-\lambda) k_1(x,y) -\lambda h_1(x) - (1-\lambda) h_1(y)\\
&= \lambda (\lambda -1) h_1(x) - (1-\lambda) (\lambda) h_1(y) + 2\lambda (1-\lambda)k_1(x,y)\\
&= -\lambda(1-\lambda) h_1(x-y)<0.
\end{split}
\end{equation*}
This concludes the proof of the fact that $f_1,f_2$ are \textit{U}-functions.

Now, condition (1) trivially holds, let us prove (2). Fix $t\in (0,1)$, we claim  that function $\tilde{u}_t=(1-t)h_1+ t h_2$ admits no minimum points. It is easy to see, by standard computation, that the functional $x^*_t \in \ell_2$ defined as
\begin{equation*}
 \left(\sum_{n=2}^{+\infty}\frac{(x_1-x_n)(1-t) + (x_n+x_1-1)t}{2^{n-1}},\dots,\frac{(1-t)(x_n-x_1)+t(x_n+x_1-1)}{2^{n-1}},\dots\right)
\end{equation*}
is the derivative of the function $\tilde{u}_t$ at the point $x=(x_1,x_2,\dots)$. By \cite[Proposition 1.26]{PhelpsBook}, $\tilde{u}_t$ has a minimum at $x$ if and only if $x_t^*=0$. Assume, on the contrary, that $x_t^*=0$, then for every $n\geq 2$ the following condition must be satisfied
\begin{equation*}
 \frac{(1-t)(x_n-x_1) + t(x_n+x_1-1)}{2^{n-1}}=0.
\end{equation*}
Which yields  $x_n=x_1-2tx_1+t$. Substituting it in the first coordinate we get
\begin{equation*}
 \sum_{n=2}^{+\infty}\frac{(1-t)(x_1 -x_1 +2tx_1-t) + t(x_1-2tx_1+t+x_1-1)}{2^{n-1}},
\end{equation*}
which is null if and only if each term of the series is equal to zero. Whence, it follows that
\begin{equation*}
\begin{split}
&(1-t)(2tx_1-t) + (2x_1 -2tx_1 +t -1)=0\\
&4tx_1-4t^2x_1+2t^2-2t=0\\
& x_1=\frac{1}{2}.
\end{split}
\end{equation*}
Which implies $x_n=\frac{1}{2}$ for every $n\in \mathbb{N}$. Since $x\in\ell_2$, we get a contradiction, and the claim is proved.

Now, suppose on the contrary that  $\|\phi(t)\|<2$. Then, since $u_t|_{2U_X}=\tilde{u}_t|_{2U_X}$, we get that $\tilde{u}_t$ attains its minimum at $\phi(t)$, a contradiction by our claim. The proof is concluded.
\end{proof}

We conclude this section by pointing out that, by means of Theorem \ref{t: continuityargmin}, we can obtain some information about the topological properties of the set of Pareto Optimum points for a multiobjective optimization problem with $f_1,\dots,f_N$  as objective functions. This type of results have been widely studied in the field of multiobjective optimization (see, e.g.,  \cite[Ch.6]{Luc} and  \cite[Sec. 3.4]{Ehrgott}). Here we just quote a corollary of our result, that immediately follows by Remark~\ref{rem: multiobjective U functions} and Theorem~\ref{t: continuityargmin}

\begin{corollary}
Let $f_1,\dots,f_N$ be a finite family of \textit{U}-functions. If $f_i$ has the SDP for every $i=1,\dots, N$, then the set of Pareto Optimum points for $f_1,\dots,f_N$ is a compact and connected subset of $X$.
    \end{corollary}

\section{Approximation of the solutions for a convex multiobjective optimization problem}\label{sec: generalizzazioneEB}
This section is devoted to develop an approximation result for the solutions of a convex multiobjective optimization problem in the spirit of the paper \cite{BE}. We start by recalling the main theorem proved in \cite{BE}.

\begin{theorem}\label{teo:EB}
Let $X$ be a uniformly convex Banach space. Let $f_1,\dots,f_N$ be a finite family of \textit{U}-functions which additionally satisfy the following conditions:
\begin{enumerate}[(a)]
\item for each $i\in\{1,\dots,N\}$, $f_i$ is G\^{a}teaux differentiable at each point of $X$.
\item for each convex bounded set $D\subset X$, there exists a constant $C_D$ and $q\geq2$ such that, for each $x_1, x_2\in D$ 
\[
 f_i\left(\frac{x_1+x_2}{2}\right)\leq \frac{1}{2} (f_i(x_1)+f_i(x_2)) - C_D \|x_1-x_2\|^q,
\]
for every $i\in \{1,\dots, N\}$.
\end{enumerate}
Then, there exist $K,\gamma>0$ such that if $t\in \Sigma_N$,  $x_t=\phi(t)$,  and $\{F_k\}$ is a sequence of independent random variables defined on the same probability space $(\Omega,\mu)$ which take the values $f_1,\dots, f_N$ with probability $t_1,\dots, t_N$, we have, for each $n\in\N$
\begin{equation}\label{eq: BE}
 \int_{\Omega}\left\|\argmin\left(\frac{1}{n}\sum_{k=1}^n F_k(\omega)\right)-x_t\right\|d\mu\leq \frac{K}{n^{\gamma}}.   
\end{equation}
\end{theorem}

\begin{remark}\label{rem:interpretazione BE multiobjective}
    From the point of view of multiobjective optimization, this result can be read as follows. First of all $x_t$ is a Pareto Optimum point for the functions $f_1,\dots,f_N$ associated to the weights $(t_1,\dots,t_N)=t\in \Sigma_N$, i.e, $x_t$ is the unique minimum point for the function $u_t$ (see Remark \ref{rem: multiobjective U functions}). Second, each realization of the random variable $\frac{1}{n}\sum_{k=1}^n F_k(\omega)$ is a convex combination of $n$ functions chosen among $f_1,\dots,f_N$ where each term has a coefficient $\nicefrac{1}{n}$. Hence, for a fixed $\omega \in \Omega$, $\argmin(\frac{1}{n}\sum_{k=1}^n F_k(\omega))$ is a Pareto Optimum point for the functions $F_1(\omega),\dots,F_n(\omega)$ where $F_j(\omega)\in \{f_1,\dots,f_N\}$. Taking into account these interpretations of the quantities appearing in the inequality (\ref{eq: BE}), we conclude that the expected value of the error made by approximating $x_t$ with a special Pareto Optimum point (the one relating to all equal weights) of a convex multiobjective optimization problem with $n$ objective functions chosen among $f_1,\dots,f_N$, where each of the $n$ chosen functions has the probability $t_j$ of being $f_j$, for each $j=1,\dots,N$, is controlled by $\nicefrac{K}{n^{\gamma}}$.

    The most interesting feature of the estimation given by Theorem \ref{teo:EB} is that each solution of the multiobjective problem with $f_1,\dots,f_N$ objective functions can be approximated in a good way by considering the minimum point of the convex combination (with equal coefficients) of $n$ functions chosen among the objective functions of the original multiobjective problem. Moreover, most of the choices of the $n$ functions among the $N$ original objective functions give a good approximation (as pointed out in \cite{BE}, see also the introduction of the present paper). Indeed, on average we provide a good estimate of the error, that depends only on the number $n$ of the chosen functions and on their geometrical properties. Finally, it is worth pointing out that this estimate does not dependent on $t \in \Sigma_N$, in particular it does not depend on the choice of the Pareto Optimum point $x_t$. 
  
\end{remark}

The aim of the present section is to present a new version of the above result under the assumption (which markedly weakens conditions (b) contained in the above theorem) that the functions $f_1,\ldots,f_N$ involved are  G\^{a}teaux differentiable  \textit{U}-functions  satisfying the SDP. On the other hand, we show that our new version of Theorem~\ref{teo:EB} fails if we drop the SDP assumption, see Example \ref{e: BEfail}. 

Let us notice that  condition (b), in Theorem~\ref{teo:EB}, implies that each $f_i$ is uniformly convex on bounded sets. On the other hand, in \cite{BE} it is claimed that  condition (b) can be replaced by the assumption that each $f_i$ is  uniformly convex on bounded sets. Since the proof of this last assertion is not  contained in \cite{BE}, we decided, for the convenience of the reader, to include  in our manuscript
the corresponding standard modifications of some preliminary results, used in \cite{BE} to prove  Theorem~\ref{teo:EB}. The next technical lemma shows that for a finite family of uniformly convex functions there exists a common modulus of convexity with some additional properties. The proof is inspired by \cite{Zali}. 

\begin{lemma}\label{l: samemodulus}
Let $f_1,\dots,f_N$ be a family of real-valued continuous functions defined on a convex subset $D$ of a Banach space $X$. Suppose that each $f_i$ is uniformly convex on $D$. Then there exists a function $\delta\colon\R^+\to \R^+$ such that
\begin{enumerate}
\item for every $x,y\in D$ and $i\in\{1,\dots,N\}$ 
\[
 f_i( \frac{x + y}{2})\leq \frac{1}{2}( f_i(x) + f_i(y)) -\delta(\|x-y\|)
\]
holds.
\item the function $t\mapsto \nicefrac{\delta(t)}{t}$ is strictly increasing on $\R^+$.
\item $\lim_{t\to 0} \nicefrac{\delta(t)}{t}=0$.
\end{enumerate}
\end{lemma}

\begin{proof}
Let $i\in \{1,\dots,N\}$. Since $f_i$ is uniformly convex on $D$, combining \cite[Remark 2.6]{Zali} with \cite[Proposition 3.5.1]{ZaliBook} there exists a map $\delta_i\colon \R^+\to \R^+$ such that:
\begin{enumerate}[(a)]
\item for every $x,y\in D$
\begin{equation*}
 f_i( \frac{x + y}{2})\leq \frac{1}{2}( f_i(x) + f_i(y)) -\delta_i(\|x-y\|)
\end{equation*}
holds;
\item the function $t\mapsto \nicefrac{\delta_i(t)}{t^2}$ is nondecreasing on $\R^+$.
\end{enumerate}
Let $s>t>0$, then we have
\begin{equation*}
    \dfrac{\delta_i(t)}{t^2}\leq \dfrac{\delta_i(s)}{s^2},
\end{equation*}
hence, it follows that
\begin{equation*}
   0< \dfrac{\delta_i(t)}{t}\leq t\dfrac{\delta_i(s)}{s^2},
\end{equation*}
which goes to zero when $t$ goes to zero. Similarly, let $s>t>0$, then we have
\begin{equation*}
    \dfrac{\delta_i(t)}{t}\leq t\dfrac{\delta_i(s)}{s^2}<\frac{\delta_i(s)}{s}.
\end{equation*}
Therefore, the function $\nicefrac{\delta_i(t)}{t}$ is strictly increasing. Finally, defining 
$$
\delta(t)=\min\{\delta_1(t),\dots,\delta_N(t)\},
$$ 
we get the desired function.
\end{proof}

Now, we are in the position of stating the main result of this section. 

\begin{theorem}\label{t: mainApplication}
Let $X$ be an uniformly convex Banach space. Let $f_1,\dots,f_N$ be a finite family of \textit{U}-functions which additionally satisfy the following conditions:
\begin{enumerate}[(a)]
\item for each $i\in\{1,\dots,N\}$, $f_i$ is G\^{a}teaux differentiable at each point of $X$.
\item for each $i\in\{1,\dots,N\}$, $f_i$ has SDP.
\end{enumerate}
Then there exists a sequence  $\{\epsilon_n\}$ of positive numbers converging to $0$, such that if $t=(t_1,\ldots,t_N)\in \Sigma_N$  and $\{F_k\}$ is a sequence of independent random variables defined on the same probability space $(\Omega,\mu)$ which take the values $f_1,\dots, f_N$ with probability $t_1,\dots, t_N$, we have, for each $n\in\N$
\[
 \int_{\Omega}\|\argmin(\frac{1}{n}\sum_{k=1}^n F_k(\omega))-\phi(t)\|d\mu\leq\epsilon_n.
\]
\end{theorem}

As we have already pointed out in Remark \ref{rem:interpretazione BE multiobjective}, also Theorem \ref{t: mainApplication} can be interpreted as an approximation result for a multiobjective optimization problem. Beside the weaker assumptions with respect to those in Theorem \ref{teo:EB}, it is worth pointing out that also the errors average upper bound $\varepsilon_n$ can be explicitly computed. Indeed, from the proof of Theorem \ref{t: mainApplication} we will see that 
$$\epsilon_n=\diam(\mathcal{M}_1)n^\beta+\eta({n^{\beta}}),$$ where $\mathcal{M}_1$ is the closed convex hull of the set of Pareto Optimum points for $f_1,\dots,f_N$, $\beta$ is a negative real constant and $\eta$ is a function vanishing as the argument goes to $0$ and depending on the geometrical properties of the objective functions.

In order to prove our main result, we recall the statements of two technical lemmata contained in \cite{BE}. Let us mention that the constant $p$ which appears in the following result is exactly the type of the Banach space $X^*$ (for the definition of type and cotype see, e.g., \cite[p. 137]{ALKA}).

\begin{lemma}\label{l: BELemma}
Let $X$ be a uniformly convex Banach space and let $A$ be a bounded subset of $X^*$. Then, there exist $p\in(1,2]$ and $C>0$ such that, 
\begin{equation*}
 \int_{\Omega}\|\frac{1}{n}\sum_{k=1}^n G_k - x^*\|d\mu\leq C n^{\frac{1}{p}-1}
\end{equation*}
whenever $x^*\in X^*$ and
$\{G_k\}$ is a sequence of independent random variables defined on the same probability space $(\Omega,\mu)$ with values in $A$, such that   $$\int_{\Omega}G_k d\mu=x^*,\qquad\qquad k\in\N.$$ 
\end{lemma}

\noindent The next result is contained in a slightly different formulation in \cite{BE}. For the sake of completeness we  include a proof.

\begin{lemma}\label{l: BELemma2}
Let $f_1,\dots, f_N$ be a finite family of real-valued \textit{U}-functions defined on $X$ and let $D\subset X$ be a convex subset. Let $t\in \Sigma_N$, $x_t=\phi(t)$, and suppose that there exists a function $\delta\colon \R^+\to \R^+$ such that for every $x,y\in D$ and $i\in\{1,\dots,N\}$
\begin{equation*}
 f_i\left(\frac{x+y}{2}\right)\leq \frac{1}{2} (f_i(x)+f_i(y)) - \delta(\|x-y\|)
\end{equation*}
holds. Then, for each $z\in D$, we have 
\begin{equation*}
 2 \delta(\|x_t-z\|)<u_t(z) - u_t(x_t).
\end{equation*}
\end{lemma}

\begin{proof} It is sufficient to observe that, for every $z\in D$, we have
\begin{equation*}
 u_t(x_t)<u_t\left(\frac{x_t+z}{2}\right)<\frac{1}{2}(u_t(x_t)+u_t(z))-\delta(\|x_t-z\|),
\end{equation*}
from which we get the assertion.
\end{proof}

\begin{proof}[Proof of Theorem \ref{t: mainApplication}]
Let us consider  the set $\mathcal{M}$ of all minimum points of convex combinations of the functions $\{f_1,\dots,f_N\}$, i.e, $\mathcal{M}=\phi(\Sigma_N)$. Since $\Sigma_N$ is a compact set and, by Theorem \ref{t: continuityargmin}, the function $\phi$ is continuous, we have that $\mathcal{M}$ is a compact set. Therefore, $\mathcal{M}_1=\cconv(\mathcal{M})$, the closure of the convex hull of $\mathcal{M}$,  is compact too (see, e.g.,  \cite[Theorem V.2.6]{Dunford Schwartz}). Hence, by Proposition \ref{p: scimpliesuc}, the functions $f_1,\dots,f_N$ restricted to the set $\mathcal{M}_1$ are uniformly convex. Therefore, by Lemma \ref{l: samemodulus}, there exists a function $\delta\colon\R^+\to \R^+$, satisfying conditions (i)-(iii) in Lemma \ref{l: samemodulus}.
If $k\in \N$, $t\in \Sigma_N$, and $x_t=\phi(t)$, we define a random variable $G_k^t$ over $(\Omega,\mu)$ with values in $X^*$ as follows: if $\omega\in\Omega$ and $F_k(\omega)=f_i$,  put $G_k^t(\omega)=f_{i}'(x_t)$. Then,  $\{G_k^t\}_k$ is a sequence of  $X^*$-valued independent random variables  taking values $f_{1}'(x_t),\ldots,f_{N}'(x_t)$ with probabilities $t_1,\ldots,t_N$. Moreover, since the functions $f_1,\ldots,f_N$ are bounded on bounded sets, it is easy to see that $\sup_{x\in \mathcal{M}_1}\max_{i}\|f'_i(x)\|$ is finite and hence there exists a bounded set $A\subset X^*$, such that  $G_k^t$ takes values in $A$, whenever $k\in\N$ and $t\in\Sigma_N$. Since, for each $k\in \N$ and $t\in\Sigma_N$, we have $\int_{\Omega}G_k^t d\mu=\sum_{i=1}^N t_i f_i'(x_t)=0$,  Lemma \ref{l: BELemma} ensures existence of constants $p\in(1,2]$ and $C>0$ such that 
\begin{equation*}
 \int_{\Omega}\|\frac{1}{n}\sum_{k=1}^{n}G_k^t \|d\mu\leq C n^{\frac{1}{p}-1},
\end{equation*}
whenever   $t\in\Sigma_N$. 
Therefore
\begin{equation*}
\mu(\Omega_1)\leq n^{\frac{1-p}{2p}},
\end{equation*}
where $\Omega_1=\{\omega\in\Omega\colon\|\frac{1}{n}\sum_{k=1}^{n}G^t_k(\omega)\|\geq C n^{\frac{1-p}{2p}}\}$. Let $\omega\in \Omega\setminus\Omega_1$ and $n\in \N$, then for each $k\leq n$, $G^t_k(\omega)=f_{i_k}'(x_t)$ for some $i_k\in \{1,\dots,N\}$. Let $x=\argmin(\frac{1}{n}\sum_{k=1}^n f_{i_k})$ and
\begin{equation*}
\psi(r)=\frac{1}{n}\sum_{k=1}^{n}f_{i_k}(x+r(x_t-x)).
\end{equation*}
The map $\psi'$ is increasing over the unit interval $[0,1]$ and attains its minimum at $r=0$. By Lemma \ref{l: BELemma2} we have
\begin{equation*}
\begin{split}
2\delta(\|x-x_t\|)&\leq \psi(1)-\psi(0)=\int_0^1\psi'(r)dr\leq |\psi'(1)|\\
&\leq\left\|\frac{1}{n}\sum_{k=1}^{n}f_{i_k}'(x_t)\right\|\|x-x_t\|\leq C n^\frac{1-p}{2p}\|x-x_t\|.
\end{split}
\end{equation*}
Defining $h(r)=\delta(r)/r$ and $\eta(r)=h^{-1}(r)$ we get
\begin{equation*}
 \frac{\delta(\|x-x_t\|)}{\|x-x_t\|}\leq \frac{Cn^{\frac{1-p}{2p}}}{2},
\end{equation*}
from which we get
\begin{equation*}
\|x-x_t\|\leq\eta\left(\frac{Cn^{\frac{1-p}{2p}}}{2}\right).
\end{equation*}
Therefore,
\begin{equation*}
\begin{split}
&\int_{\Omega}\|\argmin\left(\frac{1}{n}\sum_{k=1}^n F_k(\omega)\right)-x_t\|d\mu \leq \int_{\Omega_1}\|\argmin\left(\frac{1}{n}\sum_{k=1}^n F_k(\omega)\right)-x_t\|d\mu \\&
+ \int_{\Omega\setminus \Omega_1}\|\argmin\left(\frac{1}{n}\sum_{k=1}^n F_k(\omega)\right)-x_t\|d\mu
\leq  \diam(\mathcal{M}_1)n^{\frac{1-p}{2p}}+\eta(\frac{Cn^{\frac{1-p}{2p}}}{2}).
\end{split}
\end{equation*} If, for $n\in\N$, we define $$ \epsilon_n=\diam(\mathcal{M}_1)n^{\frac{1-p}{2p}}+\eta(\frac{Cn^{\frac{1-p}{2p}}}{2}),$$  the proof is concluded.
\end{proof}

\begin{example}\label{e: BEfail}
Let $f_1$ and $f_2$ as in Example \ref{e: argmindiscontinuo}. It is clear that both $f_1$ and $f_2$ are G\^ateaux differentiable on $X$. 
If $t=0$, then $\argmin(u_t)=\{0\}$, where $u_t= (1-t) f_1 +t f_2$. While, if $t\in (0,1]$, then $\|\argmin(u_t)\|\geq 1$. Suppose that $t\in (0,1)$. Let $(\Omega,\mu)$ be a probability space. Let us consider random variables $F_k\colon\Omega\to \{f_1,f_2\}$ that satisfy $\mu(\{\omega\in \Omega\colon F_k(\omega)=f_1\})=1-t$ and $\mu(\{\omega\in \Omega\colon F_k(\omega)=f_2\})=t$. Let 
\begin{equation*}
\begin{split}
\Omega_1&=\left\{\omega\in\Omega\colon \argmin(\frac{1}{n}\sum_{k=1}^n F_k(\omega))=\{0\}\right\}\\
&=\{\omega\in\Omega\colon F_k(\omega)=f_1 \,\forall \,k=1,\dots,n\}.
\end{split}
\end{equation*}
Then, we observe that $\mu(\Omega_1)=(1-t)^n$. Which implies,
\begin{equation*}
\begin{split}
\int_{\Omega}\|\argmin\left(\frac{1}{n}\sum_{k=1}^n F_k(\omega)\right) - x_t\|d\mu &\geq \int_{\Omega_1}\|\argmin\left(\frac{1}{n}\sum_{k=1}^n F_k(\omega)\right) - x_t\|d\mu\\
&\geq \mu(\Omega_1)=(1-t)^n.
\end{split}
\end{equation*}
Therefore, one cannot obtain an estimate which does not depend on  $t$.
\end{example}

\begin{remark}\label{rem:costruzioneCarloJacopo}
Let us point out that there exist
 G\^{a}teaux differentiable \textit{U}-functions with the SDP not satisfying condition (b), in Theorem~\ref{teo:EB}. To show this,
notice that  in \cite{DEBESOM-ALUR} is provided, in each infinite-dimensional separable Banach space (also in every reflexive space), an example of an average locally uniformly rotund, G\^{a}teaux differentiable norm $|\cdot|$ which is not locally uniformly rotund. For reflexive Banach spaces, by \cite{DEBESOM-ALUR}, a norm is average locally uniformly rotund, if and only if each $x\in S_X$ is strongly exposed by each functional which supports $x$. Hence, Proposition~\ref{t: liftingnorms} 
implies that the function $|\cdot|^2$ is 
a
 G\^{a}teaux differentiable \textit{U}-function with the SDP not satisfying condition (b), in Theorem~\ref{teo:EB}.
 
 The above construction shows that there are situations in which the hypotheses of Theorem~\ref{t: mainApplication} are satisfied but Theorem~\ref{teo:EB} cannot be applied.
\end{remark}

\end{document}